\documentclass{birkau}
\usepackage[T1]{fontenc}
\usepackage{amssymb,amsmath}
\usepackage{enumitem,url}
\usepackage{eucal}

\newtheorem{theorem}{Theorem}[section]
\newtheorem{proposition}[theorem]{Proposition}

\newtheorem{corollary}[theorem]{Corollary}
\newtheorem{fact}[theorem]{Fact}

\theoremstyle{definition}
\newtheorem{definition}[theorem]{Definition}
\newtheorem{exam}[theorem]{Example}
\newtheorem{rem}[theorem]{Remark}

\newcommand{\mb}{\mathbb}
\newcommand{\mc}{\mathcal}
\newcommand{\newspec}{\mathbf{SpecBD}}

\newcommand{\newlatz}{\mathbf{LatBD}}
\newcommand{\altspec}{\mathbf{slSpecD}^s}
\newcommand{\altlatz}{\mathbf{ZLatD}}

\newcommand{\s}{\subseteq}
\newcommand{\PI}{\mathcal{PI}}
\newcommand{\PF}{\mathcal{PF}}
\DeclareMathOperator{\ext}{\mathrm{ext}}
\DeclareMathOperator{\cl}{cl}

\begin{document}
\title{Stone Duality for  Kolmogorov  Locally Small Spaces}
\author{Artur Pi\k{e}kosz} 
\address{Department of Applied Mathematics,\\ 
Cracow University of Technology,\\
ul.  Warszawska 24, 31-155 Krak\'ow,  Poland}
\email{pupiekos@cyfronet.pl}

\begin{abstract} We prove three new versions of Stone Duality.
The main version is the following: the category of Kolmogorov  locally small spaces and  bounded continuous mappings is equivalent to 
the category of  spectral spaces with decent lumps and with bornologies in the lattices of compact (not necessarily Hausdorff) open sets as objects and  
spectral mappings respecting those decent lumps and satisfying a boundedness condition as morphisms as well as it is dually equivalent to the category of 
bounded distributive lattices with bornologies and with decent lumps of prime filters  as  objects and homomorphisms of bounded lattices respecting those decent lumps and satisfying a domination condition as morphisms. 
Some theory of strongly locally spectral spaces is  developed.

\end{abstract}

\maketitle

\section{Introduction}
Stone Duality is one of the most important dualities in mathematics.
It is very widely known for Boolean algebras and a little less known for bounded distributive lattices.
In fact,  M. H. Stone's two fundamental papers \cite{S-R,S-A} described duality between generalized Boolean algebras (or Boolean rings) and Hausdorff locally compact   Boolean spaces, where usual Boolean algebras  (or unital Boolean rings) correspond to Hausdorff compact  Boolean spaces. He achieved a beautiful theory of ideals in Boolean rings and
a beautiful theory of representations of Boolean rings in powersets.
The case of distributive lattices was considered by M. H. Stone  in \cite{S}.
Many versions of this duality  exist (see, for example,  \cite{Er} or \cite{Ha} for further literature), including versions of Priestley Duality proved by H. Priestley in \cite{Pr} with many consequences developed in \cite{Pr2}.
Stone Duality for bounded distributive lattices, while considered already in a much broader context in \cite{G}, has been presented in detail in a recent  monograph  \cite{DST} by M. Dickmann,  N. Schwartz and M. Tressl. 

In this paper,  three new versions of  Stone Duality  are proved: for small spaces, for locally small spaces with usual morphisms (bounded continuous mappings) and for 
locally small spaces with bounded strongly continuous mappings as morphisms.
In each of the cases, the Kolmogorov separation axiom ($T_0$) is assumed.

Locally small spaces may be understood to be a special kind of generalized topological spaces in the sense of Delfs and Knebusch (\cite{P}), which in turn are a special form of Grothendieck topologies (see \cite{DK,P1}) or $G$-topologies of \cite{BGR}.
 Locally small spaces were used in o-minimal homotopy theory (\cite{DK,P3}).
 A simpler language for locally small spaces
was introduced and used in \cite{P2} and \cite{P}, compare also \cite{PW}. We continue developing the theory of locally small spaces in this simple language, analogical to  the language of  Lugojan's generalized topology (\cite{Lu}) or 
Cs\'{a}sz\'{a}r's generalized topology (\cite{C}),
where a family of subsets of the underlying set is 
satisfying some, but not all, conditions for a topology. 

The main result of the paper reads as follows: the category of Kolmogorov  locally small spaces and  bounded continuous mappings is equivalent to  
the category of  spectral spaces with distinguished decent lumps and with 
 bornologies in the lattices of compact (not necessarily Hausdorff) open sets as objects
and  spectral mappings respecting those decent lumps and satisfying a boundedness condition as morphisms
and is dually equivalent to 
the category of  bounded distributive lattices with bornologies and with 
decent lumps of prime filters  as  objects
and homomorphisms of bounded lattices satisfying a domination condition and
respecting those decent lumps as morphisms.

Small spaces are a special case of locally small spaces, with some compactness flavour.
While we meet small spaces as these underlying definable spaces over structures with topologies, we meet locally small spaces as those underlying analogical locally definable spaces (\cite{P2,P}).
We show that a Kolmogorov  small space is essentially a patch dense subset of a spectral space.
More precisely: the category of Kolmogorov small spaces and continuous mappings is equivalent to
the category of spectral spaces with distinguished patch dense subsets and 
spectral mappings respecting those patch dense subsets
and is dually equivalent to
the category of bounded distributive lattices with distinguished patch dense sets of prime filters and 
homomorphisms of bounded lattices respecting those patch dense sets.
This means that spectralifications of a Kolmogorov  topological space may be constructed by choosing bounded sublattice bases of the topology.

We have another version of Stone Duality: for Kolmogorov locally small spaces with bounded strongly continuous mappings. This category is equivalent to
the category of  strongly locally spectral spaces with distinguished patch dense subsets as objects and  
strongly spectral mappings respecting those patch dense subsets as morphisms 
and is dually equivalent to  the category of
distributive lattices with zeros and distinguished patch dense sets of prime filters as objects and 
lattice homomorphisms respecting zeros and those patch dense sets and satisfying a condition of domination as morphisms. 
Moreover, some theory of strongly locally spectral spaces is  developed.

The paper is organized in the following way: Section 2 introduces  categories $\mathbf{SS}_0$ and  $\mathbf{LSS}_0$,
Section 3 deals with $\mathbf{SpecD}$ and $\newspec$, Section 4 introduces  $\mathbf{LatD}$ and  $\newlatz$. Section 5 gives the main theorem for $\mathbf{LSS}_0$ and a version for  $\mathbf{SS}_0$. Section 6 deals with spectralifications of Kolmogorov spaces. Section 7 introduces the categories $\mathbf{slSpec}$ and $\mathbf{slSpec}^s$. Section 8 deals with $\mathbf{ZLat}$ and establishes a dual equivalency between $\mathbf{slSpec}^s$ and $\mathbf{ZLat}$. Section 9 concludes with Stone Duality for $\mathbf{LSS}^s_0$. Examples throughout the paper illustrate the topic.

Regarding  the set-theoretic axiomatics for this paper, we follow Saunders Mac Lane's version of Zermelo-Fraenkel axioms  with the axiom of choice plus the existence of a set which is a universe (\cite[p. 23]{MacLane}). 

We shall freely use the notation for family intersection and family difference, 
compatible with \cite{P1,P2,P,PW}:
$$ \mc{U}  \cap_1 \mc{V}= \{ U \cap V\mid U \in \mc{U}, V\in \mc{V}    \}, \quad
 \mc{U}  \setminus_1 \mc{V}= \{ U \setminus V \mid U \in \mc{U}, V\in \mc{V}\} . $$

\section{The Categories  $\mathbf{SS}_0$ and  $\mathbf{LSS}_0$}

\begin{definition}[{{\cite[Definition 2.1]{P}}}]\label{lss}
A \textit{locally small space} 
  is a pair $(X,\mc{L}_X)$, where $X$ is any set and  $\mc{L}_X\subseteq  \mc{P}(X)$ satisfies the following conditions:
\begin{enumerate}[align=left]
\item[(LS1)] \quad $\emptyset \in \mc{L}_X$,
\item[(LS2)] \quad if $A,B \in \mc{L}_X$, then $A\cap B, A\cup B \in \mc{L}_X$,
\item[(LS3)] \quad $\forall x \in X \: \exists A_x \in \mc{L}_X  \:  x\in A_x$ (i.e., $\bigcup \mc{L}_X = X$).
\end{enumerate}
Elements of $\mc{L}_X$ are called \textit{small open} subsets  (or  \textit{smops})  of $X$.
\end{definition}

\begin{definition}[{{\cite[Definition 2.21]{P}}}]
A \textit{small space} is such a locally small space $(X,\mc{L}_X)$ that $X\in \mc{L}_X$.
\end{definition} 

\begin{definition}
A locally small space  $(X,\mc{L}_X)$ will be called $T_0$ (or \textit{Kolmogorov}) if the family $\mc{L}_X$ separates points 
 (\cite[Remainder 1.1.4]{DST}), which  means that for $x,y \in X$ the following condition is satisfied: 
$$\mbox{ if }x\in A  \iff y\in A\mbox{ for each }A \in \mc{L}_X,\mbox{ then }x=y.$$ 
\end{definition}

\begin{definition}[{{\cite[Definition 2.9]{P}}}]
If $(X,\mc{L}_X)$ is a locally small space, then the topology
$ \mc{L}_X^{wo}=\tau(\mc{L}_X), \mbox{ generated by $\mc{L}_X$ in $\mc{P}(X)$,}$
is called the family of \textit{weakly open sets} in $(X,\mc{L}_X)$. 
\end{definition}

\begin{fact}
For a small space $(X,\mc{L}_X)$, the following conditions are equivalent:
\begin{enumerate}
\item[$(1)$] $(X,\mc{L}_X)$ is $T_0$,
\item[$(2)$] the topological space $(X,\mc{L}_X^{wo})$ is $T_0$.
\end{enumerate}
\end{fact}

\begin{exam}   \label{spaces}
(1) The small spaces 
$\mb{R}_{om}=(\mb{R},\mc{L}_{om})$, 
$\mb{R}_{rom}=(\mb{R},\mc{L}_{rom})$, 
$\mb{R}_{slom}=(\mb{R},\mc{L}_{slom})$, 
$\mb{R}_{st}=(\mb{R},\tau_{nat})$ from \cite[Example 2.14]{P}, (compare 
\cite[Definition 1.2]{PW})
have the natural topology $\tau_{nat}$ on $\mb{R}$ as the topology of weakly open sets, so they are Kolmogorov  small spaces. 
In the above, we have: 
\begin{enumerate}
\item[i)]  $\mc{L}_{om}=$ the family of all  finite unions of open intervals,  
\item[ii)] $\mc{L}_{rom}=$ the family of all finite unions of open intervals with rational numbers or infinities as endpoints, 
\item[iii)] $\mc{L}_{slom}=$ the family of all locally finite (in the traditional sense) unions of bounded open intervals.  
\end{enumerate}

(2) The space $(\mb{R}, \mc{L}_{iom})$, where $\mc{L}_{iom}$ is the family of all finite unions of open intervals with integers or infinities as ends, is not  Kolmogorov.

\end{exam}

\begin{definition}[\cite{P}]
For a locally small space $(X,\mc{L}_X)$, we define the family of \textit{open} sets as
$$  \mc{L}_X^o =\{ M\subseteq  X \mid M\cap_1 \mc{L}_X \subseteq \mc{L}_X \}.$$
\end{definition}
\begin{rem}
The family $\mc{L}_X^o$ is a bounded sublattice of $\mc{P}(X)$ 
containing $\mc{L}_X$. The open sets are those subsets of $X$ that are "compatible with" smops. 
\end{rem}

\begin{exam} \label{lom}
Consider the following families of subsets of the set $\mb{R}$ of real numbers:
\begin{enumerate}
\item[i)]  $\mc{L}_{lom}=$ the family of all  finite unions of bounded open intervals, 
\item[ii)] $\mc{L}^o_{lom}=\mc{L}_{slom}=$ the family of all locally finite unions of bounded open intervals.
\item[iii)]  $\mc{L}_{lrom}=$ the family of all  finite unions of bounded open intervals with rational endpoints, 
\item[iv)] $\mc{L}^o_{lrom}=$ the family of all locally finite unions of  open intervals with rational endpoints.
\end{enumerate}
Then $\mb{R}_{lom}=(\mb{R},\mc{L}_{lom})$ and  
$\mb{R}_{lrom}=(\mb{R},\mc{L}_{lrom})$
are Kolmogorov locally small spaces (compare \cite[Example 2.14]{P}  
and \cite[Definition 1.2]{PW}) that are not small.
\end{exam}

\begin{definition}
Assume $(X, \mc{L}_X)$ and $(Y,\mc{L}_Y)$  are locally small spaces.
Then a mapping $f:X \to Y$ is: 
\begin{enumerate}
\item[(a)]  \textit{bounded} (\cite[Definition 2.40]{P}) if $\mc{L}_X$  refines $f^{-1}(\mc{L}_Y)$, 
which means that  each $A\in  \mc{L}_X$ admits $B\in \mc{L}_Y$ such that $A \s f^{-1}(B)$,
\item[(b)]  \textit{continuous} (\cite[Definition 2.40]{P})  if 
$f^{-1}(\mc{L}_Y) \cap_1 \mc{L}_X \subseteq \mc{L}_X $ (i.e., $f^{-1}(\mc{L}_Y) \subseteq \mc{L}^o_X $),
\item[(c)] \textit{strongly continuous} if $f^{-1}(\mc{L}_Y)  \subseteq \mc{L}_X$. 
\end{enumerate} 
\end{definition}
\begin{definition} We consider the following categories:
\begin{enumerate}
\item[(a)]
the  category \textbf{LSS} of  locally small spaces and their bounded continuous mappings (\cite[Remark 2.46]{P}), 
\item[(b)]
the full subcategory  $\mathbf{LSS}_0$  in  $\mathbf{LSS}$   of $T_0$ locally small spaces, 
\item[(c)] 
the full subcategory \textbf{SS} in $\mathbf{LSS}$ of  small spaces (\cite[Remark 2.48]{P}),
\item[(d)]
the full subcategory   $\mathbf{SS}_0$  in $\mathbf{LSS}$ of $T_0$ small spaces.
\end{enumerate}
\end{definition}

\section{The Categories $\mathbf{SpecD}$ and $\newspec$}

\begin{definition}
For any topological space $X=(X,\tau_X)$, we consider the following families of subsets:
\begin{enumerate}
\item[(a)]
 the family  $CO(X)$ of compact (not necessarily Hausdorff) open subsets of~$X$,
\item[(b)]
 the family $ICO(X)$  of intersection compact open subsets of $X$. (An open subset $Y$  of $X$ is \emph{intersection compact open} if for every compact open set $V$ their intersection $V\cap Y$ is compact, see \cite{E}.)
\end{enumerate}
\end{definition}

\begin{definition}
A \textit{spectral space}  is a topological  space $X=(X,\tau_X)$
satisfying the following conditions  (compare   \cite[Definition 1.1.5]{DST}):
\begin{enumerate}[align=left]
\item[(S1)] $X\in CO(X)$,
\item[(S2)] $CO(X)$ is a basis of $\tau_X$,
\item[(S3)] $CO(X) \cap_1 CO(X) \subseteq CO(X)$,
\item[(S4)] $(X,\tau_X)$ is  $T_0$,
\item[(S5)] $(X,\tau_X)$ is sober (this means: $\forall \: W\in \tau_X\setminus \{X\} \:\exists \: V_1 ,V_2 \in \tau_X \: ( W=V_1 \cap V_2 \wedge ( W\neq V_1 \wedge W\neq V_2)) \vee \: \exists x\in X  \: (W=\ext\{x\})$).
\end{enumerate}
Here $\ext V$ denotes the exterior $X\setminus \cl V$   of a set $V\s X$.
\end{definition}

Hochster (\cite{H}) proved that every spectral space is homeomorphic to the
Zariski spectrum of some commutative unital ring. 

\begin{definition}
A mapping $g:X\to Y$ between spectral spaces is \textit{spectral} if the preimage of any compact open subset of $Y$ is a compact open subset of $X$, shortly: $g^{-1}(CO(Y))\subseteq CO(X)$, see   \cite[Definition 1.2.2]{DST}.
We have the category \textbf{Spec} of spectral spaces and spectral mappings.
\end{definition}

\begin{rem}[The classical Stone Duality]\label{classical}
The category $\mathbf{Lat}$ of bounded distributive lattices with homomorphisms of bounded lattices is dually equivalent to the category 
\textbf{Spec}. While \cite[Chapter 3]{DST} uses contravariant functors and homomorphisms into a two-element lattice, we restate Stone Duality using covariant functors and prime filters.
Namely, we have:
\begin{enumerate}
\item The functor $\mathit{Sp}:\mathbf{Lat}^{op} \to \mathbf{Spec}$  is given by:
\begin{enumerate}
\item[a)] $Sp(L)= (\mc{PF}(L),\tau(\widetilde{L}))$ for  $L=(L, \vee, \wedge, 0,1)$
a bounded distributive lattice, 
where  $\mc{PF}(L)$  is the set of all prime filters in $L$ with topology 
$\tau(\widetilde{L})$ on  $\mc{PF}(L)$  generated by the family $\widetilde{L}$, where
$ \widetilde{L} =\{ \tilde{a} \mid a\in L\} \subseteq \mc{P}(\mc{PF}(L))$  and  
 $\tilde{a}=\{ {F}\in \mc{PF}(L)\mid a\in {F} \}$,
\item[b)] $Sp(h^{op})=h^{\bullet}$ for a homomorphism of bounded lattices 
$h:L\to M$ where, for ${G}\in \mc{PF}(M)$, we have
$$h^{\bullet}({G})= \{ a\in L\mid h(a)\in {G} \} \in \mc{PF}(L).$$
\end{enumerate} 

\item The functor $\mathit{Co}:\mathbf{Spec} \to \mathbf{Lat}^{op}$  is given by:
\begin{enumerate}
\item[a)] $Co(X) =CO(X)$ with obvious lattice operations on $CO(X)$,
\item[b)] $Co(g)=(\mc{L}g)^{op}$, where $\mc{L}g:CO(Y) \to CO(X)$ is defined by\\ 
$(\mc{L}g)(W)=g^{-1}(W)$ for a spectral $g:X\to Y$ and $W\in CO(Y)$. 
\end{enumerate}
\end{enumerate}
Then the compositions  $Sp  Co$, $Co  Sp$  are naturally isomorphic to 
the identity functors $Id_{\mathbf{Spec}}$, $Id_{\mathbf{Lat}^{op}}$, respectively.
Consequences of  the classical Stone Duality (\cite[3.2.5]{DST}) include:
\begin{enumerate}
\item[i)] the fact that each bounded distributive lattice $L=(L, \vee, \wedge, 0,1)$ is isomorphic to the lattice $ (\widetilde{L},\cup, \cap, \emptyset, \mc{PF}(L))$  of subsets of $\mc{PF}(L)$  and 
\item[ii)]  the equality $\widetilde{L} =CO(\mc{PF}(L))$.
\end{enumerate}
\end{rem}

\begin{definition}
An \emph{object of} 
$\mathbf{SpecD}$ is a pair $((X,\tau_X),X_{d})$ where $(X,\tau_X)$ 
is a spectral space  and $X_{d}$ is  a subset of $X$ satisfying:
$$\qquad \qquad   \forall \: U,V \in CO(X) \quad U\neq  V \implies U\cap X_d \neq V\cap X_d. $$
Then $X_d$ is called a \emph{decent subset} of $X$.
 
A \emph{morphism of} $\mathbf{SpecD}$ between 
$((X,\tau_X),X_{d})$ and $((Y,\tau_Y),Y_{d})$  is a spectral mapping $g:X\to Y$  between spectral spaces $(X,\tau_X)$ and $(Y,\tau_Y)$ that respects the decent subset, that is:  $g(X_d) \subseteq Y_d$.
\end{definition}

\begin{fact} If $X_d$ is a decent subset of a spectral space $(X,\tau_X)$, then 
 the lattice  $(CO(X), \cup, \cap, \emptyset,X)$  is isomorphic to  the lattice
$(CO(X)_d, \cup, \cap, \emptyset,X_d)$, where  
$$CO(X)_d= CO(X)\cap_1 X_d = \{ U\cap X_d\mid U\in CO(X)\} . $$ 
\end{fact}

\begin{rem}
If $((X,\tau_X),X_{d})$ is an object of $\mathbf{SpecD}$, then, by 
(\cite[3.2.8]{DST}),  both the spaces
$ \mc{PF}(CO(X))$  and $ \mc{PF}(CO(X)_d)$  with their spectral topologies   are homeomorphic to $(X,\tau_X)$. A point $x\in X$ corresponds to
$$ \hat{x}=\{ V\in CO(X)\mid x\in V\} \mbox{ in } \mc{PF}(CO(X)) \quad \mbox{   and to }$$
$$ \hat{x}^d=\{ U\in CO(X)_d\mid x\in U\} \mbox{ in } \mc{PF}(CO(X)_d), 
\mbox{ respectively}.$$ 
\end{rem}

\begin{definition}[{{\cite[Proposition 1.3.13]{DST}}}]
Let $(X,\tau_X )$ be a spectral space. Then the \textit{patch topology} (or the \textit{constructible topology}) on $X$ is the topology with  the family $CO(X) \setminus_1 CO(X)$ as a basis.
\end{definition}

\begin{proposition} \label{patch-dense}
For a spectral space $(X,\tau_X )$ and $X_d \s X$, the following conditions are equivalent:
\begin{enumerate}
\item[$(1)$] $X_d$ is patch dense,
\item[$(2)$] $X_d$ is decent.
\end{enumerate}
\end{proposition}
\begin{proof}
If the set $X_d$ is decent in $(X,\tau_X)$ 
and $U$ is a non-empty patch open set in $(X,\tau_X)$, then we may assume $U=A\setminus B$ with $A,B\in CO(X)$. 
Since $U=A \triangle (A\cap B)$ is non-empty, $A$ and $A\cap B$ are different in 
$CO(X)$, so
$A\cap X_d$ and $A\cap B\cap X_d$ are different in $CO(X)_d$.
This means $(A\setminus B) \cap X_d$ is non-empty. Hence $X_d$ is patch dense.

If $X_d$ is patch dense in $(X,\tau_X)$  and $A,B$ are different members of $CO(X)$, then $A\triangle B$ is a non-empty  patch open set. Hence $X_d$ intersects   $A\triangle B$ and $A\cap X_d$ is different from $B\cap X_d$ in $CO(X)_d$. This means  $X_d$ is decent in $(X,\tau_X)$. 
\end{proof}

\begin{exam}  \label{spectral}
The real spectrum of $\mb{R}[X]$, often  denoted by 
$\widetilde{\mb{R}}$ (see 7.1.4 b) and 7.2.6 in \cite{BCR}),
can be up to a homeomorphism described in the following way: 
it  contains points $r^-, r,r^+$ for each real number $r$,  the infinities $-\infty,+\infty$ and admits the obvious linear order.
As a basis of the topology on $\widetilde{\mb{R}}$, we take the family $\mc{B}$ containing: finite intervals $[r^+, s^-]=\widetilde{(r,s)}$ for $r,s\in \mb{R}$, $r<s$ and
infinite intervals $[-\infty, s^-]=\widetilde{(-\infty,s)}$, 
$[r^+,+\infty]=\widetilde{(r,+\infty)}$ for any $r,s \in \mb{R}$.

Then $CO(\widetilde{\mb{R}})$ is the family of finite unions of basic sets and the topological space $(\widetilde{\mb{R}}, \tau(\mc{B}))$ is spectral. The set $\mb{R}$ of real numbers is decent in this spectral space, so $((\widetilde{\mb{R}},\tau(\mc{B})), \mb{R} )$ is an object of $\mathbf{SpecD}$. (The operation $\widetilde{\ \cdot \ }$ mentioned in this example is an isomorphism between the Boolean algebra of semialgebraic sets in $\mb{R}$ and the Boolean algebra  of constructible sets in $\widetilde{\mb{R}}$, see   \cite[Proposition 7.2.3]{BCR}.) 

Any semialgebraic mapping $g:\mb{R} \to \mb{R}$ (i.e., $g$ has a semialgebraic graph) extends  (uniquely)  to a maping $\widetilde{g}:\widetilde{\mb{R}} \to \widetilde{\mb{R}}$
satisfying the condition $\widetilde{g}^{-1}(\widetilde{T})=\widetilde{g^{-1}(T)}$ for any semialgebraic $T\subseteq \mb{R}$,
 as in \cite[Proposition 7.2.8]{BCR}, which means that
 $\widetilde{g}:((\widetilde{\mb{R}}, \tau(\mc{B})), \mb{R} ) \to 
 ((\widetilde{\mb{R}},\tau(\mc{B})), \mb{R} ) $
 is a morphism of $\mathbf{SpecD}$.
\end{exam}

\begin{definition}
A \textit{bornology} in a bounded lattice $(L,\vee,\wedge,0,1)$ is an ideal  $B\subseteq L$ such that
$$\bigvee  B =1   .$$
\end{definition}

\begin{definition} \label{newspec}
An \textit{object of} $\newspec$ is a system $((X,\tau_X), CO_s(X),X_d)$ where
$(X,\tau_X)$ is a spectral space, $CO_s(X)$ is a  bornology in  the bounded  lattice $CO(X)$  and $X_d$ satisfies the following conditions:
\begin{enumerate}
\item[$(1)$]  $X_d\s \bigcup CO_s(X)$, 
\item[$(2)$]    $R_d: CO(X) \ni A \mapsto A\cap X_d \in CO(X)_d$
is an  isomorphism of lattices,
\item[$(3)$]  $CO(X)_d = (CO_s(X)_d)^o \subseteq \mc{P}(X_d)$.
\end{enumerate}
Such $X_d$ will be  called a \emph{decent lump} of $X$.

A \textit{morphism}  {from}  $((X,\tau_X), CO_s(X),X_d)$ {to} $((Y,\tau_Y), CO_s(Y),Y_d)$ \textit{in} $\newspec$ is such a spectral mapping 
 between spectral spaces  $g:(X,\tau_X)\to (Y,\tau_Y)$ that: 
\begin{enumerate}
\item[a)] satisfies the \textit{condition of boundedness}
$$ \forall A\in CO_s(X) \: \exists B\in CO_s(Y) \quad g(A)\subseteq B,$$
\item[b)]   {respects} the decent lump:  $g(X_d) \subseteq Y_d$.
\end{enumerate}
\end{definition}
\begin{exam} \label{many}
Each of the  spectral spaces $\PF(\mc{L}^o_{lom})$, $\PF(\mc{L}^o_{lrom})$ decomposes into two parts:  prime filters may or may not intersect $\mc{L}_{lom}$, $\mc{L}_{lrom}$, respectively.
Those elements of  $\PF(\mc{L}^o_{lom})$ that   intersect
 $\mc{L}_{lom}$ correspond  bijectively to  the elements of $\PF(\mc{L}_{lom})$. The latter set  may be topologically identified with an open set  in  $\PF(\mc{L}^o_{lom})$ or an open set 
 $ \bigcup_{r,s\in \mb{R}} [r^+,s^-] = 
\tilde{\mb{R}}\setminus \{ -\infty, +\infty\}$ in $\tilde{\mb{R}}$,
 using the notation of Example \ref{spectral}. 
On the other hand, $\PF(\mc{L}^o_{lom})$ has many other points (some of them may be constructed using  ultrafilters on the set of natural numbers).
Similar facts hold true for   $\PF(\mc{L}^o_{lrom})$.
\end{exam}

\section{The Categories $\mathbf{LatD}$ and  $\newlatz$}

\begin{definition}
\emph{Objects of}  $\mathbf{LatD}$ are pairs $(L, \mathbf{D}_L)$ with 
$L=\!(L,\vee,\wedge,0,1)$  a  bounded distributive lattice  and $\mathbf{D}_L\subseteq \mc{PF}(L)$   satisfying 
 $$\forall a,b \in L \quad a\neq b \implies \tilde{a}^d \neq \tilde{b}^d \:\: (\mbox{where }\tilde{a}^d=\{ {F}\in \mathbf{D}_L\mid a\in {F} \}=\tilde{a}\cap \mathbf{D}_L).$$
Then  $\mathbf{D}_L$ is called a  \emph{decent set of prime filters} on $L$.
  
\emph{Morphisms of} $\mathbf{LatD}$ are such homomorphisms of bounded lattices 
$h:L\to M$ that  
$h^{\bullet} (\mathbf{D}_{M}) \subseteq \mathbf{D}_L$.
\end{definition}

\begin{fact} \label{17}
If $\mathbf{D}_L$ is a decent set of prime filters of $(L,\vee,\wedge,0,1)$,
 then  the bounded lattice 
$(\widetilde{L}^d,\cup, \cap, \emptyset, \mathbf{D}_L ),  \mbox{ where }  \widetilde{L}^d=\{ \tilde{a}^d \mid a\in L\},$ 
is isomorphic to $(L,\vee,\wedge,0,1) $. Moreover, 
$\widetilde{L}^d=CO(\mc{PF}(L))\cap_1 \mathbf{D}_L $.
\end{fact}

\begin{definition} \label{newlatz}
An \emph{object of}  $\newlatz$ is a  system $(L, L_s, \mathbf{D}_L)$ with\\ 
$L=(L,\vee,\wedge,0,1)$ a bounded distributive lattice, $L_s$ a bornology in $L$ and\\ 
 $\mathbf{D}_L$ satisfying the conditions: 
\begin{enumerate}
\item[$(1)$]  $\mathbf{D}_L\subseteq \bigcup \widetilde{L_s}\subseteq \mc{PF}(L)$, 
\item[$(2)$]   $ \forall a,b \in L \quad a\neq b \implies \tilde{a}^d \neq \tilde{b}^d,$ where $\tilde{a}^d=\{ {F}\in \mathbf{D}_L\mid  a\in {F} \}$,
\item[$(3)$]  $\widetilde{L} \cap_1 \mathbf{D}_L = (\widetilde{L_s}\cap_1 \mathbf{D}_L )^o \subseteq \mc{P}(\mathbf{D}_L)$.
 \end{enumerate}
Such $\mathbf{D}_L$ will be called a  \emph{decent lump of prime filters} on $L$.
 
A  \emph{morphism of} $\newlatz$ from $(L, L_s, \mathbf{D}_L)$ to 
$(M, M_s, \mathbf{D}_M)$ is  such a homomorphism of bounded lattices  $h:L\to M$ that:
\begin{enumerate}
\item[a)]
satisfies  the \textit{condition of domination}
$$ \forall a\in M_s \: \exists b\in L_s \quad a \vee h(b) = h(b), $$
\item[b)] {respects} {the decent lump of prime filters}: $h^{\bullet} (\mathbf{D}_{M}) \subseteq \mathbf{D}_L$.
\end{enumerate}
\end{definition}

\section{Stone Duality for  $\mathbf{LSS}_0$ and  $\mathbf{SS}_0$}

\begin{proposition} \label{zaloz} Assume $(X, \mc{L}_X)$ is a  locally small space. Then 
$$\mc{L}_X \cong    \widetilde{\mc{L}_X} \cap_1 \hat{X} \mbox{ and } \mc{L}^o_X \cong    \widetilde{\mc{L}^o_X} \cap_1 \hat{X}=\widetilde{(\mc{L}_X})^o \cap_1 \hat{X}= \widetilde{(\mc{L}_X}\cap_1 \hat{X})^o \subseteq \mc{P}(\hat{X}),$$ 
where \quad  $\widetilde{\mc{L}_X} =\{ \widetilde{A}\mid  A \in \mc{L}_X\}$, \qquad
$\widetilde{A}=\{ F\in \PF(\mc{L}^o_X)\mid  A\in F\}$, 
$$\quad  \hat{X}=\{ \hat{x}\mid  x\in X\} \s \mc{PF}(\mc{L}_X^o),\mbox{ and }\quad  \hat{x}=\{ W\in \mc{L}_X^o \mid x\in W\}.$$ 
\end{proposition}
\begin{proof}
It is clear that $\widetilde{\mc{L}^o_X} \cap_1 \hat{X}\s \widetilde{(\mc{L}_X})^o \cap_1 \hat{X} \s \widetilde{(\mc{L}_X}\cap_1 \hat{X})^o$. 
Moreover, the Boolean algebras $\mc{P}(X)$ and $\mc{P}(\hat{X})$ are isomorphic,  where the sublattice $\widetilde{\mc{L}^o_X} \cap_1 \hat{X}$ corresponds to $\mc{L}^o_X$ and
the sublattice $\widetilde{\mc{L}_X} \cap_1 \hat{X}$ corresponds to $\mc{L}_X$. That is why $\mc{L}_X \cong    \widetilde{\mc{L}_X} \cap_1 \hat{X}$ and
$\mc{L}_X^o \cong  \widetilde{\mc{L}^o_X} \cap_1 \hat{X}= \widetilde{(\mc{L}_X}\cap_1 \hat{X})^o \subseteq \mc{P}(\hat{X})$. 
\end{proof}

\begin{theorem} \label{lsso}
The categories $\mathbf{LSS}_0$, $\newlatz^{op}$ and $\newspec$ are equivalent.
\end{theorem}
\begin{proof}

\textbf{Step 1:} Defining functor $\bar{R}:\newspec\to\mathbf{LSS}_0$.

We define
the \textit{restriction functor} $\bar{R}: \newspec \to \mathbf{LSS}_0$ by formulas
$$ \bar{R}((X,\tau_X),CO_s(X),X_d)=(X_d, CO_s(X)_d), \quad \bar{R}(g)=g_d,$$
where $g_d:X_d \to Y_d$ is the restriction of $g: X \to Y$ in the domain and in the codomain to the decent  lumps.
It is clear that $CO_s(X)_d$ is a  sublattice of $\mc{P}(X_d)$ with zero that covers $X_d$. 
Now $CO(X)$  separates points of $X$, since it  is a basis of the topology $\tau_X$. Hence both  $CO(X)_d$  and $CO_s(X)_d$  separate points of $X_d$.

For a morphism $g:X\to Y$ in $\newspec$, we have
$$g_d^{-1}(CO_s(Y)_d) \subseteq g^{-1}(CO(Y)_d) \cap_1 X_d \subseteq CO(X)_d\subseteq (CO_s(X)_d)^o$$ 
by $(3)$ of Definition \ref{newspec}, so $g_d:(X_d,CO_s(X)_d)\to (Y_d,CO_s(Y)_d)$ is continuous.
That  $g_d$ is a bounded mapping between locally small spaces follows from~$g$
satisfying the condition of boundedness.
Since the rest of the conditions are obvious, $\bar{R}$ is  indeed  a functor. 

\noindent \textbf{Step 2:} 
Defining functor $\bar{S}:\newlatz^{op}  \to \newspec$.

We define 
the \textit{spectrum functor} 
$\bar{S}:   \newlatz^{op}   \to \newspec$  by  
$$\bar{S}(L, L_s, \mathbf{D}_L)=((\mc{PF}(L),\tau(\widetilde{L}) ), \widetilde{L_s}, \mathbf{D}_L ), 
\quad  \bar{S}(h^{op})=h^{\bullet}, $$ 
where  $\tau(\widetilde{L})$ is as in Remark \ref{classical}  and  $h^{op}$ in  $\newlatz^{op}$ is the morphism $h$ in $\newlatz$ inverted.
The lattice $CO_s(\mc{PF}(L))=\widetilde{L_s}$ is a  bornology in $CO(\mc{PF}(L))$. Moreover,
we have an isomorphism of lattices  
$$CO(\mc{PF}(L))=\tilde{L} \ni \tilde{a} \: \mapsto \: \tilde{a}^d \in \widetilde{L}^d=CO(\mc{PF}(L)) \cap_1  \mathbf{D}_L$$
 and   $ \mathbf{D}_L\subseteq \bigcup  \widetilde{L_s}$,  by Definition \ref{newlatz}.
Now $(3)$ of Definition \ref{newspec} follows from $(3)$ of Definition \ref{newlatz},  so  $\mathbf{D}_L$ is a decent lump.

For a morphism $h:L\to M$  of $\newlatz$ we have 
$(h^{\bullet})^{-1}(\tilde{b})=\{ {G}\in \mc{PF}(M)\mid b\in h^{\bullet} {G}  \}=\widetilde{h(b)}$ for $b\in L$.
This means  $ (h^{\bullet})^{-1} ( \widetilde{L} ) \subseteq \widetilde{M}$, so
 $h^{\bullet}:\mc{PF}(M)\to \mc{PF}(L)$ is spectral,  satisfies the condition of boundedness
and respects the decent lump: $h^{\bullet}(\mathbf{D}_M)\subseteq \mathbf{D}_L$.
Since the rest of the conditions are obvious,  $\bar{S}$ is indeed a functor.

\noindent \textbf{Step 3:} Defining functor $\bar{A}:\mathbf{LSS}_0\to \newlatz^{op} $.

We define the \textit{algebraization functor} $\bar{A}:\mathbf{LSS}_0 \to \newlatz^{op}$  by 
$$\bar{A}(X,\mc{L}_X)=(\mc{L}^o_X, \mc{L}_X, \hat{X}), \quad \bar{A}(f)= (\mc{L}^o f)^{op},$$
where $ \mc{L}^o_X=(\mc{L}^o_X, \cup, \cap, \emptyset, X)$ is a bounded distributive lattice, 
 $\hat{X}=\hat{X}(\mc{L}^o_X)=\{\hat{x}\mid x\in X \} \subseteq \mc{PF}(\mc{L}^o_X)$ with 
$\hat{x}=\{A\in \mc{L}^o_X \mid x\in A \}$, 
and, for a strictly continuous mapping $f:(X,\mc{L}_X) \to (Y,\mc{L}_Y)$,  the mapping $\mc{L}^of:\mc{L}^o_Y \to \mc{L}^o_X$
 is  defined by $(\mc{L}^o f)(W)=f^{-1}(W)$ for $W\in \mc{L}^o_Y$. 
 
The lattice  $\mc{L}_X$ is a  bornology in $\mc{L}_X^o$ by the definition of $\mc{L}_X^o$. 
By the proof of Proposition \ref{zaloz},  $ (\mc{L}^o_X, \mc{L}_X, \hat{X})$ satisfies $(3)$ of Definition \ref{newlatz} and
$\hat{X}(\mc{L}^o_X)$ is a decent lump of prime filters on $\mc{L}^o_X$.
 
Moreover, $\mc{L}^o f: \mc{L}^o_Y \to \mc{L}^o_X$ is   a  morphism in $\newlatz$ as a homomorphism of bounded lattices satisfying  
$$(\mc{L}^of)^{\bullet}(\hat{x})=\{W\in \mc{L}^o_Y\mid x\in f^{-1}(W) \} =\widehat{f(x)} \mbox{, so }
(\mc{L}^o f)^{\bullet}(\hat{X})\subseteq \hat{Y},$$
with the domination condition
being the boundedness of the strictly continuous mapping $f$.
Since the rest of the conditions are obvious,  $\bar{A}$ is  indeed a functor. 

\noindent \textbf{Step 4:} The functor $\bar{R}\bar{S}\bar{A}$ is naturally isomorphic to  $Id_{\mathbf{LSS}_0}$.

We have $\bar{R}\bar{S}\bar{A}(X, \mc{L}_X)=
\bar{R}\bar{S}(\mc{L}_X^o, \mc{L}_X,\hat{X} )=
\bar{R}(\mc{PF}(\mc{L}_X^o),\widetilde{\mc{L}_X}, \hat{X})=(\hat{X},\widetilde{\mc{L}_X}^d)$, where $\widetilde{\mc{L}_X}^d=\widetilde{\mc{L}_X}\cap_1 \hat{X}$, 
and,  for a morphism $f:(X, \mc{L}_X) \to (Y, \mc{L}_Y)$ in $\mathbf{LSS}_0$, 
we have $\bar{R}\bar{S}\bar{A}(f)=((\mc{L}^o f)^{\bullet})_d$.

Define a natural transformation  $\eta$ from 
$\bar{R}\bar{S}\bar{A}$ to $Id_{\mathbf{LSS}_0}$  by 
$$\eta_X(\hat{X},\widetilde{\mc{L}_X}^d) \to (X, \mc{L}_X), \mbox{ where }    \eta_X(\hat{x})=x.$$ 
Then   $f\circ \eta_X(\hat{x})= \eta_Y \circ ((\mc{L}^o f)^{\bullet})_d (\hat{x})=f(x)$
for $\hat{x}\in \hat{X}$ and, by the obvious  isomorphism between $\mc{P}(X)$ and $\mc{P}(\hat{X})$
(compare Proposition \ref{zaloz}),
 each $\eta_X$ is an isomorphism 
 in $\mathbf{LSS}_0$, so $\eta$ is truely a~natural isomorphism.

\noindent \textbf{Step 5:}
 The functor $\bar{S}\bar{A} \bar{R}$ is naturally isomorphic to 
$Id_{\newspec}$.

We have $\bar{S}\bar{A}\bar{R}((X,\tau_X),CO_s(X) ,X_d)=\bar{S}\bar{A}(X_d,CO_s(X)_d)=$ \\
$\bar{S}((CO_s(X)_d)^o, CO_s(X)_d, \widehat{X_d}^d)=
(\mc{PF}((CO_s(X)_d)^o) ,\widetilde{CO_s(X)_d},\widehat{X_d}^d)$, \\
with the topology $\tau(\widetilde{CO(X)_d})$ on $\mc{PF}((CO_s(X)_d)^o)$, since, by Definition \ref{newspec}, we have $(CO_s(X)_d)^o = CO(X)_d$.
Here  we put  $\widehat{X_d}^d=\{ \hat{x}^d \mid  x\in X_d \}$  and 
$\hat{x}^d=\{ V\cap X_d \mid x\in V  \in CO(X)   \} \in  \mc{PF}(CO(X)_d)$ 
for  $x\in X$. 
For a morphism $g:X\to Y$ in $\newspec$,
 we have $\bar{S}\bar{A}\bar{R}(g)=
(\mc{L}^o  g_d)^{\bullet}$.

Define a natural transformation $\theta$ from $\bar{S}\bar{A}\bar{R}$ to 
$Id_{\newspec}$ by 
$$\theta_X: (\mc{PF}(CO(X)_d),\widetilde{CO_s(X)_d}, \widehat{X_d}^d) \to (X,CO_s(X), X_d)\mbox{ with } \theta_X(\hat{x}^d)=x.$$ 
Notice that $(\mc{L}^o  g_d)^{\bullet}(\hat{x}^d)=\{ W\cap Y_d \mid  g(x)\in W\in CO(Y)\} =\widehat{g(x)}^d$ for $x\in X$.
This means  $g\circ \theta_X=\theta_Y\circ (\mc{L}^o g_d)^{\bullet}$ and each $\theta_X$  satisfies
$\theta_X(\widehat{X_d}^d)=X_d$ and   $\theta_X (\widetilde{A_d}) =\theta_X (\{ \hat{x}^d \in \mc{PF}(CO(X)_d)\mid  x\in A \}) = A$ for $A\in CO(X)$, so
$\theta_X^{-1}(CO(X))=\widetilde{CO(X)_d}$ and  $\theta_X^{-1}(CO_s(X))=\widetilde{CO_s(X)_d}$. Hence $\theta$ is truely a natural isomorphism.

\noindent \textbf{Step 6:}
The functor $\bar{A}\bar{R}\bar{S}$ is naturally isomorphic to $Id_{\newlatz^{op}}$.

We get $\bar{A}\bar{R}\bar{S}(L, L_s, \mathbf{D}_L)=\bar{A}\bar{R}((\mc{PF}(L), \tau(\widetilde{L})),\widetilde{L_s},\mathbf{D}_L)= \bar{A}(\mathbf{D}_L, \widetilde{L_s} \cap_1 \mathbf{D}_L)=
( (\widetilde{L_s} \cap_1 \mathbf{D}_L)^o ,\widetilde{L_s} \cap_1 \mathbf{D}_L, \widehat{\mathbf{D}_L}^d)$. 
Here $\widehat{\mathbf{D}_L}^d=\{ \widehat{{F}}^d\mid  {F}\in \mathbf{D}_L\}$, where  $\widehat{{F}}^d=\{ \tilde{a}^d \in \widetilde{L}^d \mid {F}\in \tilde{a}^d \}$,
 $\tilde{a}^d =\{ {F}\in \mathbf{D}_L\mid  a \in {F} \}$.
 By  Definition \ref{newlatz}, we have $(\widetilde{L_s} \cap_1 \mathbf{D}_L)^o = \widetilde{L} \cap_1 \mathbf{D}_L$, shortly: $(\widetilde{L_s}^d)^o=\widetilde{L}^d$.
For a morphism  $h:L\to M$ in $\newlatz$,
we have $\bar{A}\bar{R}\bar{S}(h^{op})=(\mc{L}^o (h^{\bullet})_d )^{op}$.

Define a natural transformation  $\kappa^{op}$ from $\bar{A}\bar{R}\bar{S}$ to 
$Id_{\newlatz^{op}}$ by  putting
$\kappa^{op}_L:(\widetilde{L}^d, \widetilde{L_s}^d,\widehat{\mathbf{D}_L}^d) \to (L, L_s,\mathbf{D}_L)$ in  $\newlatz^{op}$ to be the map
$$\kappa_L:(L, L_s, \mathbf{D}_L) \to (\widetilde{L}^d,\widetilde{L_s}^d, \widehat{\mathbf{D}_L}^d)
\mbox{ given by }  
\kappa_L(a)=\tilde{a}^d.$$ 
We are to check that 
$ \kappa_L^{op} \circ  \bar{A}\bar{R}\bar{S}(h^{op}) =  h^{op}\circ \kappa_M^{op}$ or 
$\kappa_M\circ h =\mc{L}^o (h^{\bullet})_d \circ \kappa_L$.
Now $(\mc{L}^o (h^{\bullet})_d\circ \kappa_L)(a)= (h^{\bullet})_d^{-1}(\tilde{a}^d)=
\{ {G}\in \mathbf{D}_M \mid  h(a)\in {G} \}  =\widetilde{h(a)}^d=
(\kappa_M \circ h)(a)$.   
Each $\kappa_L:L\to \widetilde{L}^d$ is an  isomorphism of bounded lattices satisfying $\kappa_L^{\bullet}(\widehat{\mathbf{D}_L}^d)=\mathbf{D}_L$  and 
$\kappa_L(L_s)= \widetilde{L_s}^d$,  so $\kappa^{op}$ is truely a~natural isomorphism.
\end{proof}

\begin{exam}
The sine mapping $\sin:\mb{R}_{lom} \to \mb{R}_{lom}$ is bounded continuous but not strongly continuous.
Consequently, $\bar{S}\bar{A}(\sin)=(\mc{L}^o \sin)^{\bullet}\!: \PF(\mc{L}^o_{lom}) \to \PF(\mc{L}^o_{lom})$ is spectral but 
$(\mc{L}^o \sin ) (\mc{L}_{lom})$ is not contained in $\mc{L}_{lom}$.
\end{exam}

\begin{theorem} \label{sso}
The categories $\mathbf{SS}_0$,  $\mathbf{LatD}^{op}$ and 
$\mathbf{SpecD}$ are  equivalent.
\end{theorem}
\begin{proof} In the proof of  Theorem \ref{lsso},  we restrict to the case  $\mc{L}_X^o=\mc{L}_X$,  $L_s=L$  and $CO_s(X)=CO(X)$.
\end{proof}

\section{Spectralifications}

\begin{definition}
A topological space $(Y, \tau_Y)$ will be called a \textit{spectralification} of a topological space 
$(X, \tau_X)$ if: 
\begin{enumerate}
\item[$(1)$]  $(Y, \tau_Y)$ is a spectral space, 
\item[$(2)$]  there exists a topological embedding 
$e:(X, \tau_X) \to (Y, \tau_Y)$,
\item[$(3)$]     $e(X)$ is patch dense in $(Y, \tau_Y)$.
\end{enumerate}
\end{definition}

\begin{exam} \label{om}
The space $(\widetilde{\mb{R}}, \tau(\mc{B}))$ from Example \ref{spectral} is a spectralification of the real line (with the natural topology),  homeomorphic to  $(\mc{PF}(\mc{L}_{om}), \tau(\widetilde{\mc{L}_{om}}))$. 
\end{exam}

\begin{exam}  \label{rom}
Consider $\mc{L}_{rom}$ from Example \ref{spaces}.
The points of $\mc{PF}(\mc{L}_{rom})$  are:
$$ \hat{r}=\{A\in \mc{L}_{rom}\mid  r\in A  \} \qquad \mbox{ for }r\in \mathbb{R},$$
$$ \hat{q}^-=\{ A\in \mc{L}_{rom}\mid (l,q)\subseteq  A \mbox{ for some }l<q  \} \mbox{ for }q\in \mathbb{Q},$$
$$ \hat{q}^+=\{ A\in \mc{L}_{rom}\mid (q,l)\subseteq A   \mbox{ for some } l>q  \} \mbox{ for }q\in \mathbb{Q},$$
$$ \widehat{-\infty}=\{ A \in \mc{L}_{rom}\mid  (-\infty,l)\subseteq A \mbox{ for some } l \in \mathbb{Q}  \},$$
$$ \widehat{+\infty}=\{ A\in \mc{L}_{rom} \mid (l,+\infty)\subseteq  A   \mbox{ for some } l \in \mathbb{Q} \}.$$
Here $\widetilde{\mc{L}_{rom}}=\{ \widetilde{A}\mid A\in \mc{L}_{rom}\}, \quad 
\widetilde{A}=\{ {F}\in \mc{PF}(\mc{L}_{rom})\mid A\in {F}  \}$.     

We can see that $(\mc{PF}(\mc{L}_{rom}), \tau(\widetilde{\mc{L}_{rom}}))$ 
is another spectralification of the real line,  homeomorphic to the space of types over 
$\mathbb{Q}$  of the theory $Th(\mathbb{R}, <)$  with the spectral topology (compare \cite[Section 14.2]{DST}),
obtained without using the language of model theory. 
\end{exam}

\begin{exam}
For the patch  topology on the ``same'' set of points as in the previous example, one  takes as the new  $\mc{L}_X$
the Boolean algebra $\mc{B}_{rom}$ generated by $\mc{L}_{rom}$. This changes the topology on $\mb{R}$: the rational points become isolated.
The space $(\mc{PF}(\mc{B}_{rom}), \tau(\widetilde{\mc{B}_{rom}}))$ is a Hausdorff spectralification of this modified real line, and is identified with the (usual in model theory) space of types over 
$\mathbb{Q}$  of the theory $Th(\mathbb{R}, <)$.
\end{exam}

\begin{theorem}
For a topological space, 
being   $T_0$ is equivalent to
admitting  a spectralification.
\end{theorem}

\begin{proof} For $(X, \tau_X)$  a $T_0$ topological space,
choose a basis $\mc{L}_X$ of $\tau_X$  that is a bounded sublattice.
By Step 4 of the proof of Theorem \ref{lsso}, 
$(X, \tau_X)$  embeds into the spectral space 
$(\mc{PF}(\mc{L}_X),\tau( \widetilde{\mc{L}_X}))$ whose patch topology has
  a basis   $ \widetilde{\mc{L}_X}\setminus_1 \widetilde{\mc{L}_X}$.
Each nonempty member of this basis  $\widetilde{A} \setminus \widetilde{B}$ admits $x\in A\setminus B$. Then $\hat{x} \in \widetilde{A} \setminus \widetilde{B}$, so the image $\hat{X}$ of the embedding is patch dense in
$(\mc{PF}(\mc{L}_X),\tau( \widetilde{\mc{L}_X}))$. Since a subspace
of a $T_0$ space is $T_0$, only $T_0$ topological spaces can have spectralifications.
\end{proof}


\section{The Categories $\mathbf{slSpec}$ and $\mathbf{slSpec}^s$}

\begin{definition}
For a topological space  $(X,\tau_X)$, we denote:
$$SO(X)=\{ U\in \tau_X \mid  U \mbox{ has spectral subspace topology}\}.$$ 
\end{definition}

\begin{definition}
A topological space  $(X,\tau_X)$ is \textit{strongly locally spectral} if it satisfies the following conditions:
\begin{enumerate}
\item[$(1)$]  it is \textit{locally spectral} (\cite{E}): $ SO(X)$ covers $X$,
\item[$(2)$]  it is \textit{semispectral} (\cite{E}): $CO(X) \cap_1 CO(X) \subseteq CO(X)$.
\end{enumerate}
\end{definition}

\begin{proposition}
For any   strongly locally spectral space $(X,\tau_X)$, we have 
$$CO(X)=SO(X).$$ 
\end{proposition}
\begin{proof} Obviously, $SO(X)\subseteq CO(X)$.
Let $A\in CO(X)$. Then $A$ is covered by a finite family $W_{1},...,W_n$ of spectral open sets. 
Since a finite union of spectral spaces glued together along  compact open subsets is spectral,
 the set $W_{1}\cup ...\cup W_n$ is spectral and its compact open subset  $A$ belongs to $SO(X)$.
\end{proof}

\begin{proposition}
In a strongly locally spectral space $(X,\tau_X)$, we have
$$ CO(X)^o = ICO(X) \s \mc{P}(X).$$
\end{proposition}
\begin{proof}
If $V\in CO(X)^o$, then $V$ is a union of compact open sets since $CO(X)$ covers $X$.
Hence  $V\in \tau_X$ and $V$ satisfies the definition of a member of $ICO(X)$.

If $V\in ICO(X)$ and $A$ is any member of  $CO(X)$, then $V\cap A\in CO(X)$. This means $V\in CO(X)^o$.
\end{proof}

\begin{proposition}
A strongly locally spectral space on a set $X$ may be equivalently
defined by:
\begin{enumerate}
\item[a)] the topology $\tau_X$, 
\item[b)]  the family of compact open subsets $CO(X)$, 
\item[c)]  the family of spectral open subsets $SO(X)$, 
\item[d)]  the family of  intersection compact open subsets $ICO(X)$. 
\end{enumerate}
\end{proposition}
\begin{proof}
Elements of $CO(X)=SO(X)$ are the compact elements of $\tau_X$. Elements of $ICO(X)$ are the sets compatible with those of $CO(X)$. Elements of $\tau_X$ are the unions of subfamilies of $ICO(X)$. Each of the considered families induces all the other.
\end{proof}

\begin{definition}
A subset $X_d \s X$ in a strongly locally spectral space $(X,\tau_X)$  will be called \textit{decent} if any of the two equivalent conditions is satisfied:
\begin{enumerate}
\item[$(1)$]  for $A,B\in CO(X)$ if $A\neq B$, then $A\cap X_d \neq B\cap X_d$,
\item[$(2)$]    $R_d: CO(X) \ni A \mapsto A\cap X_d \in CO(X)_d$
 is an isomorphism of lattices.
\end{enumerate}
\end{definition}

\begin{definition}
The \textit{patch topology} of a strongly locally spectral space  $(X,\tau_X)$  is the topology on $X$  with  a basis 
$ CO(X) \setminus_1 CO(X)$.
\end{definition}

\begin{proposition} In a strongly locally spectral space  $(X,\tau_X)$ the decent subsets are exactly the patch dense subsets.
\end{proposition}
\begin{proof}
The same as the proof of Proposition \ref{patch-dense}.
\end{proof}

\begin{exam}
The spaces  $(\PF(\mc{L}_{lom}), \tau(\widetilde{\mc{L}_{lom}}))$ and
 $(\PF(\mc{L}_{lrom}), \tau(\widetilde{\mc{L}_{lrom}}))$ are strongly locally spectral  and are homeomorphic to  open  patch dense subspaces in the spectral spaces 
 $(\PF(\mc{L}^o_{lom}),\tau( \widetilde{\mc{L}^o_{lom}} ))$, $(\PF(\mc{L}^o_{lom}),\tau( \widetilde{\mc{L}^o_{lom}} ))$, respectively, that are   known from Example \ref{many}. 
\end{exam}

\begin{definition} \label{sls}
A mapping $g:(X,\tau_X) \to (Y,\tau_Y)$ between  strongly locally spectral spaces is \textit{spectral}  if the following conditions are satisfied:
 \begin{enumerate}
\item[$(1)$]  $g$ is \textit{bounded}: $g(CO(X))$ refines $CO(Y)$,
\item[$(2)$]  $g$ is \textit{s-continuous}: $g^{-1}(ICO(Y))\subseteq ICO(X)$. 
\end{enumerate}
\end{definition}

\begin{proposition} \label{morph}
If $g:(X,\tau_X) \to (Y\tau_Y)$ is a  mapping between strongly locally spectral spaces, then the following conditions are equivalent:
\begin{enumerate}
\item[$(1)$]  $g$ is spectral,
\item[$(2)$] $g$ is bounded and locally spectral 
(i.e.,   for any $A\in CO(X)$, $B\in CO(Y)$ such that $g(A)\subseteq B$,
the restriction  $g_A^B:A\to B$ is a spectral mapping between spectral spaces). 
\end{enumerate}
\end{proposition}
\begin{proof} 
$(1)\implies (2)$ If $g$ is spectral and $A$, $B$ are as in the statement, then 
$$CO(A)=ICO(A),\quad  CO(B)=ICO(B) \subseteq ICO(Y).$$ 
Now $(g_A^B)^{-1} (CO(B))\subseteq ICO(X) \cap_1 A \subseteq CO(A)$, so $g$ is locally spectral.

\noindent $(2)\implies (1)$ If $g$ is bounded and locally spectral, then, for $D\in ICO(Y)$ and $A\in CO(X)$, we have 
$$g^{-1}(D) \cap A=(g_A^B)^{-1}(D\cap B) \in CO(X),$$
with some $B\in CO(Y)$  such that $g(A)\subseteq B$. This means $g$ is s-continuous.
\end{proof}

\begin{definition}
By \textbf{slSpec} we shall denote the category of  strongly locally spectral spaces and  spectral mappings between them.
\end{definition}

\begin{definition} \label{slss}
A mapping $g:(X,\tau_X)\to (Y,\tau_Y)$ between  strongly locally spectral spaces is \textit{strongly spectral}  if the following conditions are satisfied:
 \begin{enumerate}
\item[$(1)$]  $g$ is \textit{bounded},
\item[$(2)$]  $g$ is \textit{strongly continuous}: $g^{-1}(CO(Y))\subseteq CO(X)$. 
\end{enumerate}
\end{definition}

\begin{definition}
By $\mathbf{slSpec}^s$ we shall denote the category of  strongly locally spectral spaces and  strongly spectral mappings between them.
\end{definition}

\section{The Category $\mathbf{ZLat}$}

\begin{definition}
For a homomorphism of lattices $h:L\to M$, we say that $h$ is \textit{dominating} or
satisfies the \textit{condition of domination}, if:
$$ \forall a\in M \: \exists b\in L \quad a \vee h(b) = h(b). $$
\end{definition}

\begin{definition}
By $\mathbf{ZLat}$ we denote the category of distributive lattices with zeros and  dominating homomorphisms of lattices respecting zeros. 
\end{definition}

\begin{definition}
For    $L$  a distributive lattice with zero  and $a\in L$,  we set 
$$D^f(a)=\{  q\in \PF(L)\mid  a \in q \}, \qquad  D(a)=\{  p\in \PI(L)\mid a \notin p \}, $$
$$V^f(a)=\{  q\in \PF(L)\mid  a \notin q \}, \qquad  V(a)=\{  p\in \PI(L)\mid a \in p \}, $$
where $\PI(L)$ is the set of all prime ideals in $L$.
\end{definition}

Below, we restate  in a modern language  two theorems  of M. H. Stone published in 1938.

\begin{theorem}[{{\cite[Theorem 15]{S}}}]  \label{stone1}
Let $(L, \vee, \wedge, 0)$ be a distributive lattice with zero. Then the sets
$$ D(I)=\{ p \in \PI(L)\mid  I \not\subseteq p \}, \quad \mbox{where $I$ is an ideal in }L, $$
form a $T_0$  topology on $ \PI(L)$ with a basis
$$ \{  D(a)\mid  a\in L \} = CO(\PI(L))$$ 
closed under finite intersections and satisfying the condition \\
 $(\star)$  for a closed set $F$ and a subfamily $ \mathcal{C} \subseteq CO(\PI(L)) $
 centered on $F$ (this means: for any finite family $C_1,...,C_n$ of members of $\mathcal{C}$ the set $F\cap C_1 \cap ... \cap C_n$ is nonempty), the intersection
 $F\cap \bigcap \mathcal{C}$ is nonempty.
\end{theorem}

\begin{theorem}[{{\cite[Theorem 16]{S}}}] \label{stone2}
Let $(X, \tau_X)$ be a topological $T_0$-space where $CO(X)$ is a basis of the topology closed under finite intersections and satisfying the condition  $(\star)$ from the previous theorem.
Then: 
\begin{enumerate}
\item[$(1)$] $(CO(X), \cup, \cap, \emptyset )$  is a distributive lattice with zero,
\item[$(2)$] $\Psi: \mathcal{I}(CO(X)) \ni I \mapsto \bigcup I \in \tau_X$ is an isomorphism of lattices,
where $\mathcal{I}(CO(X))$ is the lattice of all ideals in $CO(X)$,
\item[$(3)$] for each $p \in \PI(CO(X))$ there exists a unique $x_p\in X$ such that 
$$\bigcup p = \ext \{ x_p \},$$ 
\item[$(4)$]  the mapping  $H:\PI(CO(X)) \ni p \mapsto x_p \in X$ is a homeomorphism, where the topology in $\PI(CO(X))$ is defined as in Theorem \ref{stone1}.    
\end{enumerate}
\end{theorem}

The following proposition  gives an explanation to  Theorem \ref{stone2}.

\begin{proposition} \label{rem}
For a topological space $(X,\tau_X)$, the following conditions are equivalent:
\begin{enumerate}
\item[$(1)$] $(X,\tau_X)$ is  strongly locally spectal,
\item[$(2)$] $(X,\tau_X)$ satisfies the conditions in the assumption of Theorem $\ref{stone2}$.  
\end{enumerate}
\end{proposition}
\begin{proof}
$(1) \implies (2)$  Since the other conditions are obvious, we  prove $(\star)$. One may assume $F\neq \emptyset$
and  $\emptyset \neq \mathcal{C} \subseteq CO(X) $ is centered on $F$. Choose $C\in \mc{C}$. Then $F\cap C$ is patch compact and  members of $\mc{C}\cap_1 C$ are patch closed in $C$. Since finite subfamilies of $\mc{C}\cap_1 C$ meet $F\cap C$, the set $F\cap \bigcap \mathcal{C}$   is nonempty. 
 
\noindent $(2)\implies (1)$ Assume $V$ is a proper irreducible open subset 
of $X$. Then $I(V)=\{ A\in CO(X)\mid  A\s V \}$ is a prime  ideal in $CO(X)$.
By  (3)  of Theorem \ref{stone2}, there exists a unique $x_V$ such that 
$\bigcup I(V)=V = \ext \{ x_V\}$.  Hence $X$ as well as all members of $CO(X)$ are sober.  The other conditions are obvious. 
\end{proof}

The next corollary is a restatement of Theorem \ref{stone1}.

\begin{corollary} \label{48}
If $(L, \vee, \wedge, 0)$ is a distributive lattice with zero, then both \\
 $(\PF(L), \tau(\{ D^f(a) \mid a\in L\} ))$ and $(\PI(L), \tau(\{ D(a) \mid a\in L\}))$  
 are strongly locally spectral topological spaces. 
\end{corollary}
\begin{proof} Follows from Theorem \ref{stone1} and Proposition \ref{rem}.
\end{proof}

\begin{theorem} \label{neq}
The categories $\mathbf{slSpec}^s$ and $\mathbf{ZLat}^{op}$ are equivalent. 
\end{theorem}
\begin{proof}
\textbf{Step 1:} Defining functor $\widehat{Co}:\mathbf{slSpec}^s \to \mathbf{ZLat}^{op}$.

For an object  $(X, \tau_X)$ of $\mathbf{slSpec}^s$, we set 
$\widehat{Co}(X, \tau_X)  =(CO(X), \cup, \cap, \emptyset)$. 
For a morphism $g:(X,\tau_X)\to (Y,\tau_Y)$ of  $\mathbf{slSpec}^s$, we set $\widehat{Co}(g)=(\mc{L} g)^{op}: CO(X)\to CO(Y)$, where $(\mc{L}g)(W)=g^{-1}(W)$ for $W\in CO(Y)$  defines a morphism $\mc{L}g$ of $\mathbf{ZLat}$.
Hence  $\widehat{Co}$ is a well defined functor.

\noindent \textbf{Step 2:} Defining functor $\widehat{Sp}:   \mathbf{ZLat}^{op} \to \mathbf{slSpec}^s$.

For an object $L=(L, \vee, \wedge, 0)$ of $\mathbf{ZLat}$, we put  
$\widehat{Sp}(L)=(\PF(L),\tau(\tilde{L}))$, where $\tilde{L}=D^f(L)=\{ D^f(a)\mid  a\in L\}$, 
which is a strongly locally spectral space by Corollary \ref{48} (or Theorem~\ref{stone1}).
For a morphism $h:L\to M$ of $\mathbf{ZLat}$, 
we set $\widehat{Sp}(h^{op})=h^{\bullet}$, which is a  strongly spectral mapping. \\
Boundedness of $h^{\bullet}$: since $h^{\bullet} (\widetilde{h(L)})$ refines $\widetilde{L}$ and $\widetilde{h(L)}=\{D^f(h(a))\mid a\in L\}$ dominates in $\widetilde{M}$, hence $h^{\bullet} (\widetilde{M})$ refines $\widetilde{L}$. \\
Strong continuity of $h^{\bullet}$: for any $\widetilde{a}\in CO(\PF(L))=\{ D^f(a)\mid  a\in L\}$ (see  Theorem \ref{stone1}), we have  $(h^{\bullet})^{-1} (\widetilde{a})= \widetilde{h(a)} \in CO(\PF(M))$. \\
Hence $\widehat{Sp}$ is a well defined functor.

\noindent \textbf{Step 3:}  The functor  $\widehat{Co} \widehat{Sp}$   is naturally isomorphic to $Id_{\mathbf{ZLat}^{op}}$.

Define a natural transformation $\alpha$ from $Id_{\mathbf{ZLat}^{op}}$ to $\widehat{Co} \widehat{Sp}$ by
 $\alpha_M(a)=\tilde{a}=D^f(a) \in \widetilde{M}$
for any object $M$ of $\mathbf{ZLat}$. Then each $\alpha_M:M \to \widetilde{M}$  is an  isomorphism of $\mathbf{ZLat}$
(injectivity follows from \cite[Theorem 6]{S}).
For a morphism $h:L\to M$ in $\mathbf{ZLat}$, one has 
$(\alpha_M \circ h) (a)= \widetilde{h(a)}= ( h^{\bullet})^{-1} (\widetilde{a})
= (\mc{L}h^{\bullet} \circ \alpha_L)(a)$, so $\alpha_L \circ h^{op}= (\mc{L} h^{\bullet})^{op}\circ \alpha_M$.
 Hence $\alpha$ is a natural isomorphism.

\noindent \textbf{Step 4:}  The functor  $\widehat{Sp} \widehat{Co}$   is naturally isomorphic to $Id_{\mathbf{slSpec}^s}$.

Define a natural transformation $\beta$ from $Id_{\mathbf{slSpec}^s}$ to $\widehat{Sp} \widehat{Co}$  by $\beta_X(x)=\hat{x}$  for any object $(X,\tau_X)$ of $\mathbf{slSpec}^s$, where $\hat{x}=\{ V\in CO(X)\mid x\in V \}$.
(We have $\widehat{X}=\{ \hat{x}\mid x\in X\} =\PF(CO(X))$ by the dual of Theorem \ref{stone2}).
For a morphism $g:(X,\tau_X)\to (Y,\tau_Y)$ of $\mathbf{slSpec}^s$, we have 
$(\beta_Y \circ g)(x)=\widehat{g(x)}= (\mc{L}g)^{\bullet} (\widehat{x})=( (\mc{L}g)^{\bullet} \circ \beta_X)(x)$ for $x\in X$.
Now $\beta_X$  is an isomorphism, since $\beta_X(CO(X))=CO(\hat{X})$, where
 $CO(\hat{X})=\{ \tilde{A}\mid A\in CO(X)\}$ and $\tilde{A}=\{\hat{x}\mid x\in A\}$. Hence $\beta$ is a natural isomorphism.
\end{proof}

\section{Stone Duality for $\mathbf{LSS}^s_0$}

\begin{definition} \label{altspec}
The category $\altspec$ has as objects pairs $((X,\tau_X), X_d)$ where $(X,\tau_X)$ is a
strongly locally spectral space and $X_d$ is a distinguished decent subset of $X$
and as morphisms strongly spectral mappings respecting the decent subsets.
\end{definition}

\begin{definition} \label{altlatz}
The category $\altlatz$ has pairs $(L, \mathbf{D}_L)$ where $L$ is a distributive lattice with zero and $\mathbf{D}_L$ is a distinguished decent set  
of prime filters in $\PF(L)$ as objects   and 
 homomorphisms of lattices with zeros respecting the decent sets of prime filters and satisfying the condition of domination as morphisms.
\end{definition}

\begin{definition}
The category $\mathbf{LSS}^s_0$ is a subcategory of $\mathbf{LSS}_0$ with the same objects and
{bounded strongly continuous} mappings as morphisms.
\end{definition}

\begin{exam} \label{strongly}
Let $\pi: \mb{R}_{lom} \sqcup \mb{R}_{lom}\to \mb{R}_{lom}$ be the natural projection
from the disjoint union of two copies of the real locally o-minimal line to the real locally o-minimal line.
This finite covering mapping is an example of a bounded strongly continuous mapping but not an isomorphism of  $\mathbf{LSS}^s_0$.
\end{exam}

\begin{theorem}
The categories $\mathbf{LSS}^s_0$, $\altlatz^{op}$ and $\altspec$ are equivalent.
\end{theorem}
\begin{proof}
Similar to  the proof of Theorem \ref{lsso}, using Theorem \ref{neq} instead of the classical Stone Duality, with no necessity to mention explicitely the ambient bounded lattice  $\mc{L}_X^o$ of $\mc{L}_X$, with an object $L$ of  $\mathbf{ZLat}$  playing the role of $L_s$ and $CO(X)_d$ playing the role of $CO_s(X)_d$ in Theorem \ref{lsso}, restricting to the appropriate classes of morphisms.
\end{proof}

\begin{exam}
The mapping $(\mc{L} \pi)^{\bullet}\! : \!\mc{PF}(\mc{L}_{lom} \oplus \mc{L}_{lom})\!\to \mc{PF}(\mc{L}_{lom})$, where 
 $\pi: \mb{R}_{lom} \sqcup \mb{R}_{lom}\to \mb{R}_{lom}$ 
is as in Example \ref{strongly}, is the natural projection from the disjoint union of two copies of 
$\mc{PF}(\mc{L}_{lom})$ to $\mc{PF}(\mc{L}_{lom})$. It is a strongly spectral mapping between strongly locally spectral spaces. Moreover, 
 $\hat{\mb{R}}(\mc{L}_{lom})$ is a patch dense set in  $\mc{PF}(\mc{L}_{lom})$
 and 
  $\widehat{\mb{R}\sqcup \mb{R}}(\mc{L}_{lom}\oplus \mc{L}_{lom})$ is a patch dense set in  $\mc{PF}(\mc{L}_{lom}\oplus \mc{L}_{lom} )$.  The morphism $(\mc{L} \pi)^{\bullet}$  of  $\altspec$ corresponds to the morphism $\pi$ of $\mathbf{LSS}^s_0$
and may be understood as an extention of $\pi$.
\end{exam}


\begin{thebibliography}{DST}

\bibitem{C} 
Cs\'{a}sz\'{a}r, \'{A}.: {Generalized topology, generalized continuity}. {Acta Math. Hungar.} \textbf{96}(4),  351--357  (2002)

\bibitem{DK}
 Delfs,  H.,  Knebusch, M.: {Locally Semialgebraic Spaces}. Lecture Notes in Math., vol. {1173}. 
Springer, Berlin-Heidelberg (1985)

\bibitem{BCR} 
Bochnak,  J., Coste, M.,  Roy, M.-F.: {Real Algebraic Geometry}. Ergeb. Math., Folge 3,  vol. 36. Springer, Berlin (1998)

\bibitem{BGR}
Bosch, S., G\"{u}ntzer,  U., Remmert, R.: {Non-Archimedean Analysis}. Grundlehren Math. Wiss., vol. {261}.
Springer, Berlin-Heidelberg (1984)

\bibitem{DST} 
Dickmann, M.,  Schwartz, N., Tressl, M.: {Spectral Spaces}. Cambridge University Press, Cambridge (2019)

\bibitem{E} 
Echi, O., Abdallahi, M.O.: On the spectralification of a hemispectral space. {J. Algebra  Appl.} 
\textbf{10}(4),  687--699  (2011) 

\bibitem{Er} 
Ern\'{e}, M.: General Stone duality. {Topology Appl.} \textbf{137}, 125--158 (2004) 
 
\bibitem{G}  Gr\"{a}tzer, G.: Lattice theory: First Concepts and Distributive Lattices. 
 W.H. Freeman, San Francisco (1971)
 
\bibitem{Ha} 
Hartonas,   C.: Stone duality for lattice expansions. {Log. J. IGPL} \textbf{26}(5),
 475--504 (2018)  
 
\bibitem{H} 
Hochster,  M.: Prime ideal structure in commutative rings. 
{Trans. Amer. Math. Soc.} \textbf{142},  43--60 (1969)

\bibitem{MacLane}
Mac Lane, S.: {Categories for the Working Mathematician}. Grad. Texts in Math., vol.  {5}. Springer, New York (1998)

\bibitem{Lu} 
 Lugojan, S.: Generalized topology. {Stud. Cerc. Mat.} \textbf{34},
  348--360 (1982)

\bibitem{P1}
 Pi\k{e}kosz,  A.: {On generalized topological spaces I}. 
 {Ann. Polon. Math.} \textbf{107}(3),  217--241 (2013) 

\bibitem{P2}
Pi\k{e}kosz, A.: {On generalized topological spaces II}. 
{Ann. Polon. Math.} \textbf{108}(2), 185--214 (2013) 

\bibitem{P3}  
Pi\k{e}kosz,  A.: {O-minimal homotopy and generalized (co)homology}.
{Rocky Mountain J. Math.} \textbf{43}(2),  573--617 (2013) 

\bibitem{P} 
Pi\k{e}kosz, A.: {Locally small spaces with an application}. 
{Acta Math. Hungar.}  \textbf{160}(1),  197--216  (2020)

\bibitem{PW}
 Pi\k{e}kosz,  A., Wajch, E.: {Compactness and compactifications in generalized topology}. 
{Topology Appl.} \textbf{194}, 241--268 (2015)  

\bibitem{Pr} 
Priestley, H.: Representation of distributive lattices
by means of ordered Stone spaces. 
{Bull. London Math. Soc.} \textbf{2},  186--190 (1970) 

\bibitem{Pr2} 
Priestley, H.: Ordered topological spaces and the
representation of distributive lattices.
{Proc. London Math. Soc.}  \textbf{24},  507--530  (1972)

\bibitem{S-R} 
 Stone,  M.H.:  The theory of representations for Boolean algebras. 
 {Trans. Amer. Math. Soc.}  \textbf{40},   37--111  (1936)

\bibitem{S-A} 
 Stone, M.H.:
Applications of the theory of Boolean rings to general topology.
{Trans. Amer. Math. Soc.}  \textbf{41},  375--481 (1937)  

\bibitem{S}  
 Stone, M.H.: Topological representations of distributive lattices and Brouwerian logics. {\v{C}asopis P\v{e}st. Mat. Fiz.} \textbf{67},  1--25   (1938)

\end{thebibliography}
\end{document}